\documentclass[12pt]{article}
\usepackage[utf8]{inputenc}
\usepackage{amssymb,amsmath,amsthm}
\usepackage{color}
\definecolor{cFF0000}{RGB}{255,0,0}
\newtheorem{theorem}{Theorem}[section]
\newtheorem{definition}{Definition}[section]

\newtheorem*{question}{Question}
\newtheorem{lemma}[theorem]{Lemma}
\newtheorem{proposition}[theorem]{Proposition}
\newtheorem{corollary}[theorem]{Corollary}

\begin{document}

\title{Moving to Extremal Graph Parameters}

\author{P.J.~Cameron \\ Queen Mary University of London \and C.A.~Glass \\ City University London \and   R.~Schumacher\thanks{EPSRC Grant EP/P504872/1} \\ City University London}

\maketitle

\begin{abstract}
Which graphs, in the class of all graphs with given numbers $n$ and $m$ of
edges and vertices respectively, minimizes or maximizes the value of some
graph parameter? In this paper we develop a technique which provides answers
for several different parameters: the numbers of edges in the line graph,
acyclic orientations, cliques, and forests. (We minimize the first two and
maximize the third and fourth.)

Our technique involves two moves on the class of graphs. A \emph{compression}
move converts any graph to a form we call \emph{fully compressed}: the fully
compressed graphs are split graphs in which the neighbourhoods of points in
the independent set are nested. A second \emph{consolidation} move takes each
fully compressed graph to one particular graph which we call $H_{n,m}$.
We show monotonicity of the parameters listed for these moves in many cases, which enables us to obtain our results fairly simply.

The paper concludes with some open problems and future directions.
\end{abstract}

\section{Introduction}

In this paper we consider graphs with given numbers of vertices and edges,
and investigate which of them have the smallest numbers of acyclic orientations
or edges in the line graph, or the greatest numbers of cliques or forests.
We develop a technique which handles all these parameters, and possibly others.

It is known that calculating the exact values for the number of acyclic
orientations and the number of forests is \#P-hard \cite{F.Jaeger1990}.
However, we are able to find the extremal graph for these parameters.
The result for the number of acyclic orientations was found by Linial
\cite{lowerboundspecial}, based on work by Stanley \cite{Stanley}.
The power of our method is the unified approach to several parameters.

In section \ref{sec:compression_a}, we define an operation called 
\emph{compression} for reducing graphs to a standard structure.
In Section \ref{sec:extremalgraphs}, we classify the \emph{fully compressed
graphs}, those which cannot be further compressed using this move. These are
\emph{split graphs} (the vertex set is the disjoint union of a clique and an
independent set) in which the neighbourhoods of vertices in the independent
set are nested. The structure of these graphs allows us to apply a further 
\emph{consolidation} move in Section \ref{sec:compression_b} to obtain a
single graph $H_{n,m}$, in which a clique and a set of isolated vertices
together contain all but at most one vertex of the graph.

We will first show that compression and consolidation terminate, and then in
Section \ref{sec:other_parameters} show that both moves are monotone in
several graph parameters. The terminal graph $H_{n,m}$ we obtain for
these moves is therefore also extremal for each of the graph parameters.

\section{Background}

First we give some basic definitions which will be used throughout the paper. Much of this notation is the standard notation used in the literature. A graph will usually be denoted by $G = (V,E)$, and will have $n=|V|$ vertices and $m= |E|$ edges. We assume graphs to be simple i.e. there are no multiple edges and no loops, as this is required for the compression move.

\begin{definition}
The \emph{neighbourhood} of a vertex $x$, $\mathcal{N}(x)$ is the set of vertices of $G$ that are connected to $x$ by an edge in $G$, i.e. $\mathcal{N}(x) = \{ v \in V(G) | (x,v) \in E(G) \}$.
\end{definition}

\begin{definition}
An \emph{orientation} of a graph $G$ is an assignment of direction to each edge of $G$. Let \emph{$(G, \theta)$} denote the graph $G$ equipped with the orientation $\theta$.
\end{definition}

\begin{definition}
Let \emph{$\mathcal{O}(G)$} be all orientations of $G$ and \emph{$\mathcal{C}(G)$} be the orientations of $G$ that contain at least one cycle respectively.
\end{definition}

\begin{definition}
A \emph{Euler subgraph} is a subgraph in which every vertex has even degree.
\end{definition}

\section{Compression} \label{sec:compression_a}

\begin{definition}
A \emph{compression} is a function on the set of all graphs. The compression associated with a graph $G$ from $x$ to $y$ for $x,y \in V(G)$, denoted by $C_{xy}(G)$, moves all edges from $(v,x)$ to $(v,y)$ for $v \in V$ that it is possible to move without creating any double edges, preserving orientation if one exists.
\end{definition}

Formally compression constitutes the following mapping on edges in $G$:

\begin{equation*}
  (a,b) \mapsto \begin{cases}
    (y,b), & \text{if $a=x, b \neq y, b \notin \mathcal{N}(y) $}.\\
    (a,y), & \text{if $b=x, a \neq y, a \notin \mathcal{N}(y) $}.\\
    (a,b), & \text{otherwise}.
  \end{cases}
\end{equation*}

Note that this map may be applied both to a directed and an undirected graph $G$. Sometimes we will refer to the orientation of the compressed graph as the \emph{compression of an orientation}. Moreover, $G$ when compressed by $C_{xy}$ is isomorphic to $G$ compressed by $C_{yx}$ under the map interchanging $x$ and $y$ and fixing all other vertices. Similarly, we refer to a compression $C_{xy}(G)$ as \emph{isomorphic} if $C_{xy}(G)$ is isomorphic to $G$. However, the graph we obtain from $G$ after compressing repeatedly is not unique. Consider the following small example.

\setlength{\unitlength}{0.5mm}
\begin{picture}(80,90)
\put(30,5){\makebox(0,0){$G$}}
\put(20,20){\line(0,1){50}}
\put(45,70){\line(1,-1){25}}
\put(45,20){\line(1,1){25}}
\put(20,70){\circle*{3}}
\put(20,20){\circle*{3}}
\put(45,70){\circle*{3}}
\put(45,20){\circle*{3}}
\put(70,45){\circle*{3}}
\put(10,70){$u$}
\put(10,20){$v$}
\put(50,70){$x$}
\put(50,20){$y$}
\put(75,45){$z$}
\end{picture}
\begin{picture}(80,90)
\put(30,5){\makebox(0,0){$A$}}
\put(20,20){\line(2,1){50}}
\put(45,70){\line(1,-1){25}}
\put(45,20){\line(1,1){25}}
\put(20,70){\circle*{3}}
\put(20,20){\circle*{3}}
\put(45,70){\circle*{3}}
\put(45,20){\circle*{3}}
\put(70,45){\circle*{3}}
\put(10,70){$u$}
\put(10,20){$v$}
\put(50,70){$x$}
\put(50,20){$y$}
\put(75,45){$z$}
\end{picture}
\begin{picture}(80,90)
\put(30,5){\makebox(0,0){$B$}}
\put(45,20){\line(0,1){50}}
\put(45,70){\line(1,-1){25}}
\put(45,20){\line(1,1){25}}
\put(20,70){\circle*{3}}
\put(20,20){\circle*{3}}
\put(45,70){\circle*{3}}
\put(45,20){\circle*{3}}
\put(70,45){\circle*{3}}
\put(10,70){$u$}
\put(10,20){$v$}
\put(50,70){$x$}
\put(50,20){$y$}
\put(75,45){$z$}
\end{picture}

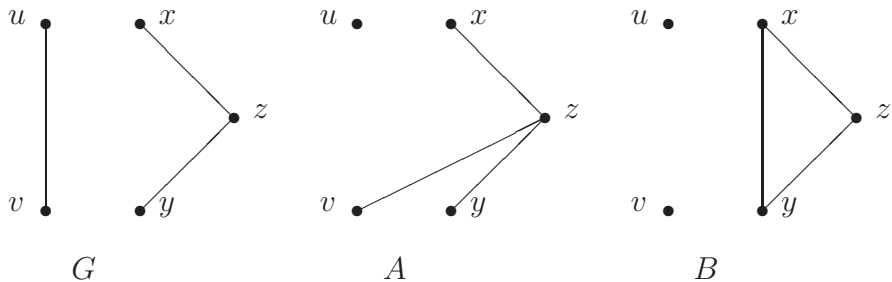
\textbf{Fig. 1.} An example of two different compressions of a graph.

We can compress from $G$ to $A$ in one move, $C_{uz}(G) = A$. From $G$ to $B$ we compress twice, $C_{ux}(C_{vy}(G)) = B$. Neither of these two graphs can be compressed further.

\subsection{Partial Order}
For the remainder of this paper we consider $n$ and $m$ fixed, $m \leq \binom{n}{2}.$ We consider a graph whose vertices are the graphs up to isomorphism with $n$ vertices and $m$ edges, and whose edges identify a compression from the graph of one vertex to the graph of the other vertex. Furthermore, we add a direction to each edge in the direction of the compression. We shall show that this orientation is acyclic, so its transitive closure is a partial order on all graphs with respect to compression. In order to prove this we need to consider the line graphs of these graphs.

\begin{definition}
Given a graph $G$, the \emph{line graph} $L(G)$ is the graph whose vertices are the edges of $G$, and two vertices $x$ and $y$ in $L(G)$ are connected iff the corresponding edges share a common endpoint in $G$.
\end{definition}

This is standard notation. We will be interested in the number of edges of the line graph of $G$, $|E(L(G))|$, which we will simplify to $e_{L(G)}$. We will show that the number $e_{L(G)}$ strictly increases with each compression that gives us a graph that is not isomorphic to the original graph. This in turn implies that compression must terminate, as the number of edges of the line graph of a graph with $n$ vertices and $m$ edges is bounded.

\begin{lemma} \label{lem:comp_a_link_line}
For a graph $G$ and a non-isomorphic compression $C_{xy}(G)$ we have $e_{L(G)} < e_{L(C_{xy}(G))}$.
\end{lemma}

\begin{proof}
Let $x$ and $y$ be vertices in $G$ and $H:=C_{xy}(G)$ and $x'$ and $y'$ be the vertices $x$ and $y$ in $H$. By counting each edge twice in the line graph we obtain the following formula for $ e_{L(G)}$:

\[ e_{L(G)} = \frac{1}{2} \sum_{v \in V(G)} \binom{d(v)}{2} .\]

The function $f(x) = \binom{x}{2}$ is a strictly convex function, and as $d(x)+d(y) = d(x') + d(y')$, and the difference $|d(x)-d(y)|$ is maximised by compression, we have that

\[ \binom{d(x)}{2} + \binom{d(y)}{2} < \binom{d(x')}{2} + \binom{d(y')}{2}.
\]
Since $d(v)$ is unchanged for $v \ne x,y$, the result follows from the above expression for the values of $e_{L(G)}$ and $e_{L(H)}$.
\end{proof}

\begin{corollary} \label{lem:partial_order}
The compression move gives us a partial order on the set of all graphs with $n$ vertices and $m$ edges.
\end{corollary}

\begin{lemma}
The complement of a compressed graph is identical to the compression of a complement, i.e.
\[ C_{xy}(\overline{G}) = \overline{C_{xy}(G)} \]
\end{lemma}

\begin{proof}
Moving all edges that can be moved in one direction is the same as moving all non-edges that can be moved in the opposite direction.
\end{proof}

\subsection{An extremal set of graphs} \label{sec:extremalgraphs}

We see that compression gives us a partial order on the set of all graphs with $n$ vertices and $m$ edges. It is natural to ask which graphs are at the bottom and which at the top of this ordering. We first wish to classify so called \emph{fully compressed graphs} to which we cannot apply any further non-isomorphic compression. It is helpful here to introduce the following definition.

\begin{definition}
A \emph{split graph} is a graph in which the vertices can be partitioned into a clique and an independent set.
\end{definition}

\begin{definition}
Let \emph{$H(k,n_0,n_1,...,n_{k-1})$} be the graph with a complete subgraph $K_k$, and with $n_i$ extra vertices of degree $i$ whose neighbourhoods are in $K_k$ and are nested.
\end{definition}

This definition identifies a graph uniquely up to isomorphism. Observe that $n = k + \sum n_i$ and $m =  \sum i n_i + \binom{k}{2}$ and that there are $n_0$ isolated vertices. Below in Fig 2 we have an example of such a graph with $G = H(5,1,1,0,2)$.

\newpage

\setlength{\unitlength}{0.5mm}
\begin{picture}(80,70)
\put(10,10){\line(1,0){70}}
\put(10,10){\line(0,1){25}}
\put(10,35){\line(1,0){70}}
\put(80,10){\line(0,1){25}}
\put(20,25){$K_5$}
\put(45,20){\circle*{3}}
\put(15,20){\circle*{3}}
\put(30,20){\circle*{3}}
\put(60,20){\circle*{3}}
\put(75,20){\circle*{3}}

\put(45,50){\circle*{3}}
\put(75,50){\circle*{3}}
\put(60,50){\circle*{3}}
\put(90,50){\circle*{3}}

\put(45,20){\line(0,1){30}}
\put(60,20){\line(0,1){30}}
\put(75,20){\line(0,1){30}}

\put(45,20){\line(1,2){15}}

\put(60,20){\line(-1,2){15}}
\put(75,20){\line(-1,2){15}}

\put(75,20){\line(-1,1){30}}

\end{picture}

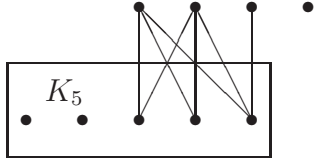
\textbf{Fig. 2.} Graph $H(5,1,1,0,2)$.

This graph is a special split graph in that the neighbourhoods of the vertices in the independent set are nested in the clique. It is equivalent to say that the neighbourhoods of the vertices in the clique are nested in the independent set. Note that the complement of a split graph is also a split graph and that indeed the complement of a fully compressed graph is also a fully compressed graph.

\begin{theorem}\label{thm:compressed}
The set  of all graphs of the form $H(k, n_0, ... , n_{k-1})$ is precisely the set of fully compressed graphs.
\end{theorem}

\begin{proof}

We first show that every graph in $\mathcal{H} = H(k,n_0,n_1,...,n_{k-1})$ is fully compressed. Any compression in $\mathcal{H}$ to a vertex in the complete subgraph does not change the graph. Any compression from a vertex in the complete subgraph to another vertex not in the complete subgraph gives us an isomorphic graph. Any compression from a vertex not in the complete subgraph to one not in the complete subgraph either does not change the graph, or gives us an isomorphic graph. So indeed every graph in $\mathcal{H}$ is fully compressed.

Now suppose we have any fully compressed graph $G$. By definition $G$ gives an isomorphic graph under any compression. It follows that there can only be one component of $G$ that has any edges, else any compression from any component to any other component reduces the number of components by one and we obtain a non-isomorphic graph. Now we pick a largest clique in $G$, say of size $k$, with the highest sum of degrees of its vertices. There can be no edge in $G$ that is not connected directly to this clique. If there were such an edge, then simply by compressing one end-point of this edge into the clique (this is possible - else we would have started with a larger clique) gives us a higher sum of degrees in this clique, which can not be isomorphic to our original $G$. Suppose now that the neighbourhoods of two points not in this clique are not nested. Then a compression from one to the other will give us a non-isomorphic graph, so indeed the neighbourhoods are nested. This completes the structure of graphs in $\mathcal{H}$.
\end{proof}

\begin{corollary}\label{cor:H_maximises}
The graph $G$ with $n$ vertices and $m$ edges that maximises the number of edges of the line graph of $G$ is in the set $\mathcal{H}$.
\end{corollary}

\begin{proof}
By Lemma \ref{lem:comp_a_link_line} we know that compression increases the number of edges of the line graph, and every graph can be compressed to a graph in $\mathcal{H}$ by Theorem \ref{thm:compressed}. 
\end{proof}

\section{Consolidation} \label{sec:compression_b}

We have thus far shown that repeatedly using the compression move, results in a graph in the set $\mathcal{H}$. We now introduce a second move, a \emph{consolidation move}, which applies only to graphs in $\mathcal{H}$. This move will allow us to consolidate further to obtain a single graph with $n$ vertices and $m$ edges.

\begin{definition}\label{def:comp_b}
A \emph{consolidation} on a graph $G = H(k,n_0,n_1,...,n_{k-1}) \in \mathcal{H}$ is the following move: If $\sum _{i=1} ^{k-1} n_i \leq 1$ then consolidation does nothing. Otherwise take a vertex $x$ of smallest non-zero degree $s$, and remove an edge connected to this vertex. Take a vertex $y$ of largest degree smaller than $k$, $l$, and add an edge connected to this vertex and the clique of size $k$ in $G$.
\end{definition}

Formally consolidation acts as follows on $H(k,n_0,n_1,...,n_{k-1}) \in \mathcal{H}$ if $\sum _{i=1} ^{k-1} n_i \geq 2$ for $i = 1,...,k-1$:

\begin{equation*}
  n_i \leftarrow \begin{cases}
    n_i + 1, & \text{if $i = s-1$ or $i = l+1$}.\\
    n_i - 1, & \text{if $i=s$ or $i=l$}.\\
    n_i, & \text{otherwise}.
  \end{cases}
\end{equation*}
If $l = k-1$ then $k \leftarrow k+1$ and we introduce $n_k = 0$.

This move is well-defined, as graphs in $\mathcal{H}$ are defined by the numbers $k, n_o,...,n_{k-1}$ alone. It takes a graph in $\mathcal{H}$ and maps it to a graph in $\mathcal{H}$. Consider the simple example $G=H(5,0,0,1,2,0)$. After one consolidation move we will have $H(5,0,1,0,1,1)$. The next move gives $H(6,1,0,0,1,0,0)$. This final graph is not affected by the consolidation move as $\sum _{i=1} ^{5} n_i = 1$.

\begin{lemma} \label{lem:consolidation_terminates}
We cannot repeatedly apply consolidation without obtaining isomorphic graphs. The only graphs which are not affected by consolidation are those with at most one vertex connected to the clique.
\end{lemma}

\begin{proof}
If $\sum _{i=1} ^{k-1} n_i \geq 2$ then $\sum_{x \in V(G)} d(x)^2$ is strictly increased by consolidation, as the sums of the degrees remain constant, and only two vertices have a change in degree. A pair of vertices have degrees $s$ and $l$ respectively before consolidation, and $s-1$ and $l+1$ afterwards. As $s \leq l$ and $f(x) = x^2$ is convex the sum of squares must strictly increase. $\sum_{x \in V(G)} d(x)^2$ is bounded for given $m$ and $n$ which completes the proof of the first statement. The second statement follows directly from the definition.
\end{proof}

It is in fact possible to obtain a partial order on the set $\mathcal{H}$ using consolidation. This order is similar to (can be extended to) the order that is lexicographic first on the number of isolated vertices and then on the smallest degree.

\subsection{The extremal graph} \label{sec:the_extremum}

Now that we have shown that consolidation also terminates, we are interested in the final graph. For $m = \binom{k}{2}$ we define $H_{n,m}:=H(k,n-k,0,...,0)$, otherwise $H_{n,m}:=H(k,n-k-1,...,0,1,0,...,0)$, where $k$ is determined by $\binom{k}{2} < m < \binom{k+1}{2}$, $d_{m - \binom{k}{2}}=1$ and $d_i = 0$ otherwise for $i>0$. $H_{n,m}$ has one clique, and one vertex (of degree $m - \binom{k}{2}$) connected only to the clique, and $n-k-1$ isolated vertices.

\section{Application to cliques, forests and acyclic orientations} \label{sec:other_parameters}

We have shown that the two transformations, compression followed by consolidation, move us to a unique graph with $n$ vertices and $m$ edges $H_{n,m}$. If for any invariant, we can show demonstrate monotonic behaviour with respect to both these moves, i.e. is increasing or decreasing, then the extremal value of that graph parameter is obtained by the graph $H_{n,m}$. If we can further show that the parameter is strictly increasing or strictly decreasing, then $H_{n,m}$ is the unique extremal graph. We consider the following parameters:

\begin{itemize}
\item the number of cliques of $G$
\item the number of forests of $G$
\item the number of acyclic orientations of $G$.
\end{itemize}

\subsection{Monotonicity for compression}

We have already shown that for the number of edges of the line graph compression is strictly monotone (in Lemma \ref{lem:comp_a_link_line}) and that the extremal graph must lie in $\mathcal{H}$ (see Corollary \ref{cor:H_maximises}). In Lemmas \ref{lem:compressiondoesnotdecreasecliques}, \ref{lem:compressiondoesnotdecreaseforests} and \ref{lem:compressionlemma} we show that compression is monotone in the number of cliques, Euler subgraphs and acyclic orientations respectively. The proof for the first two is quite short, but proving monotonicity for acyclic orientations is harder and requires several other lemmas.

\begin{lemma} \label{lem:compressiondoesnotdecreasecliques}
Compression cannot decrease the number of cliques in a graph. In fact the number of cliques of every size does not decrease.
\end{lemma}

\begin{proof}
Take a graph $G$ and vertices  $x,y \in G$. It is sufficient to find an injection from the set of cliques of $G$ to the set of cliques of $C_{xy}(G)$. The only cliques in $G$ which can be affected by compression are those that contain $x$ and not $y$. Take a clique consisting of the vertices $(a_1, ..., a_k, x)$ say. If $(a_i, y)$ is an edge for all $a_i$, then this clique is unaffected by compression. Otherwise some edge $(a_i, x)$ in $G$ is mapped to $(a_i, y)$ by compression and all vertices $a_i$ are connected to $y$ after compression. Thus we can map the original clique in $G$ to the clique containing the vertices $(a_1,...,a_k,y)$ in $C_{xy}(G)$. We map each other clique to itself. No two cliques get mapped to the same clique. Hence we have an injection and we are done.
\end{proof}

\begin{definition}
Given a graph $G$ and two vertices $x,y \in G$ we let $E^{xy}(G) = \{ (u, x)| u \in \mathcal{N}(x) \backslash \mathcal{N}(y) \}$.
\end{definition}

The edges $E^{xy}(G)$ are precisely the edges that get moved by the compression $C_{xy}$. Thus any cycle in $G$ that gets destroyed by compression must go through such an edge $(x,u)$.

\begin{lemma} \label{lem:samedirectionlemma}
Given an oriented graph $G$ and two vertices $x,y \in G$. For every vertex $u \in E^{xy}(G)$ for which $(u,x)$ is contained in a cycle in $G$ but $(u,y)$ is not contained in a cycle in $C_{xy}(G)$, there exists a vertex $v$ such that $uxvy$ forms a directed path in $G$.

\end{lemma}

\begin{proof}
Take such a vertex $u$. There is a cycle $uxvP$ for some vertex $v$ and path $P$, which is destroyed by compression. Suppose that the vertex $v$ is not connected to $y$ in $G$. Then the compression moves the edge $(x,v)$ as well as $(u,x)$ and the cycle is preserved. So $v$ must be connected to both $x$ and $y$. Furthermore, if $(y,v)$ has the same orientation as $(x,v)$, then $uyvP$ is a cycle in $C_{xy}(G)$ contradicting the property of $u$. So there must exist at least one vertex $v$ as described in the lemma.
\end{proof}

\begin{corollary}
The paths $xvy$ described in Lemma \ref{lem:samedirectionlemma} must all point in the same direction from $x$ to $y$ in $G$.
\end{corollary}

\begin{proof}
Suppose we have two paths $xv_1y$ and $yv_2x$ in those directions. Then we have the cycle $xv_1yv_2$. This is a contradiction.
\end{proof}

\begin{lemma} \label{lem:compressiondoesnotdecreaseforests}
Compression $C_{xy}$ does not decrease the number of forests of $G$.
\end{lemma}

\begin{proof}
Pick a subgraph $H \in G$ such that $H$ is a forest and $C_{xy}(H) \in C_{xy}(G)$ contains a cycle. For each such $H$ we need only find a corresponding $H'$ such that $H'$ contains a cycle and $C_{xy}(H')$ is a forest to show that the number of forests does not decrease.

Such a graph $H$ contains a unique cycle $uxvP$ with $(u,y) \notin E(G)$ and $(v,y) \notin E(H)$, but $(v,y) \in E(G)$. The vertices $u$ and $v$ exist and are unique, else there would be a cycle in $C_{xy}(H)$. Now let $H' = H - (x,v) + (y,v)$, thus $C_{xy}(H') = C_{xy}(H) - (x,v) + (y,v) \in C{xy}(G)$. $H'$ contains no cycle, if it did, then $C_{xy}(H)$ would contain a cycle. $C_{xy}(H')$ contains the cycle $uyvP$ so we are done.
\end{proof}

We now show that a compression does not increase the number of acyclic orientations. The proof requires some lemmas, which aid the construction of an orientation used in the proof of the main result, the Compression Theorem for acyclic orientations Theorem \ref{lem:compressionlemma}.

\begin{definition}
A \emph{path switching} $P_{ab}(G)$ in an acyclic graph $G$ is the reversal of direction on every edge that is on a directed path between $a$ and $b$.

\end{definition}

\begin{lemma}[Path Switching Lemma] \label{lem:pathswitchinglemma}
For a given acyclic graph $G$ and any two vertices $a,b \in G$, $P_{ab}(G)$ is acyclic.
\end{lemma}

\begin{proof}
Suppose otherwise. Then there is a graph $G$ which is acyclic, while $P_{ab}(G)$ contains a cycle. Some of the cycle must have a proper subpath which has been path-switched, else $G$ would not be acyclic. Hence the cycle can be viewed as a sequence of subpaths, $Q_1,P_1,Q_2,P_2,...$ in $G$, where the $Q_j$ are path switched to create a cycle in $P_{ab}(G)$ and the $P_j$ are not. In $G$ these path segments point in alternating directions, in $P_{ab}(G)$ they all point in the same direction, forwards say, to create a cycle. For every edge $e$ in a path segment $Q_j$ there must exist a path from one end to $a$, and from the other end to $b$. All of these paths must point in the same direction from $a$ to $b$, else there would be a cycle in $G$. Thus there exists a path in $G$ from $a$ to $b$ containing $Q_j$, $A_j Q_j B_j$ say for each $Q_j$, $q \ge 1$. If there exists a $Q_j$ for $j>1$, then we note that $A_1 P_1 B_2$ is a path in $G$ from $a$ to $b$ containing $P_1$, a contradiction since $P_1$ is not switched by $P_{ab}$. Thus, the cycle in $P_{ab}(G)$ is of the form $Q_1P_1$.  But then $A_1 P_1 B_1$ is a path from $a$ to $b$ containing $P_1$ providing the required contradiction.
\end{proof}

  We now use path switching and the set of edges $E^{xy}$ to define a map $D$.

\begin{definition}
Given a graph $G$ and two vertices $x,y \in G$ the map ${D \! \! : \:} \mathcal{O}(G) \to \mathcal{O}(C_{xy}(G))$ acts as follows: For $\theta$ such that $G - E^{xy}$ is acyclic, $D$ removes the edges moved by compression, applies path-switching and then puts the edges back in, i.e. \[ D_{xy}(G) = P_{xy}(G - E^{xy}) \cup E^{xy}. \]
Otherwise $D_{xy}(\theta) = \theta$.
\end{definition}

Note that $D$ is invertible and is its own inverse.

\begin{definition} \label{def:mapb}
The map ${B \! \! : \:} \mathcal{O}(G) \to \mathcal{O}(C_{xy}(G))$ acts as follows on orientations of $G$: For an orientation $\theta$ with a cycle in $G$ for which $C_{xy}(\theta)$ is acyclic, we define for $\theta ' = \theta$ and $\theta ' = D_{xy}(\theta)$, $B(\theta') = C_{xy}(D_{xy}(\theta'))$. For all other orientations we let $B(\theta) = C_{xy}(\theta)$.
\end{definition}

\begin{lemma} \label{lem:bisabijection}
For given $G$, $x,y \in G$ the map $B_{xy}$ is a bijection.
\end{lemma}

\begin{proof}
$|\mathcal{O}(G)| = |\mathcal{O}(C_{xy}(G)|$, so it is sufficient to prove that the map $B$ is an surjection. Consider the orientations of $G$ in pairs, $\theta$ and $D_{xy}(\theta)$, allowing for cases where $\theta = D_{xy}(\theta)$. Then $B$ acts either as $C_{xy}$ or as $C_{xy} \circ D_{xy}$ on both orientations in the pair, by construction. Thus, since $D$ is its own inverse, $ \{ B(\theta), B(D_{xy}(\theta)\} = \{ C_{xy}(\theta), C_{xy}(D_{xy}(\theta))\}$. The map $ C_{xy}$ is a surjection and the image is the same as the image of $B$ so we have a surjection.
\end{proof}

\begin{corollary} \label{cor:binjsur}
The map $B$ is an injection of orientations with cycles to orientations with cycles, and a surjection of acyclic orientations to acyclic orientations.
\end{corollary}

\begin{proof}
By the construction of $B$ we have an injection of orientations with cycles to orientations with cycles, thus it follows from Lemma \ref{lem:bisabijection} that we must have a surjection of acyclic orientations to acyclic orientations.
\end{proof}

\begin{theorem}[Compression Theorem for acyclic orientations] \label{lem:compressionlemma}
For any graph $G$, and any $x, y \in V(G)$, applying $C_{xy}$ to $G$ cannot increase the number of acyclic orientations, i.e.

\[ a(G) \ge a(C_{xy}(G)).\]

\end{theorem}

 \begin{proof}
Take a graph $G$ and two vertices $x,y \in V(G)$. Take the map ${B \! \! : \:} \mathcal{O}(G) \to \mathcal{O}(C_{xy}(G))$ defined in Definition \ref{def:mapb}. In Lemma \ref{lem:bisabijection} we show that $B$ is a bijection, so it is sufficient to show that every orientation with a cycle gets mapped to an orientation with a cycle.

Consider the case of $\theta$ cyclic such that $C_{xy}(\theta)$ is acyclic. Then $B(\theta) = C_{xy}(D_{xy}(\theta))$ by the definition of $B$. There exists a cycle $uxvP^{vu}$ in $G$ for some $(u,x) \in E_{xy}$ by Lemma \ref{lem:samedirectionlemma}. Furthermore there exists a vertex $q$ on path $P^{vu}$ that is the last vertex in $P^{vu}$ on a path from $x$ to $y$ ($v$ is in $P^{vu}$ so there exists at least one). Call the path from $q$ to $u$ $P^{qu}$ and the path from $q$ to $y$ $P^{qy}$. Path switching affects $P^{qy}$ and gives us the path $P^{yq}$, but does not affect $P^{qu}$. Now in $C_{xy}(D_{xy}(\theta))$ we have the cycle $uyP^{yq}P^{qu}$.

Next consider the case of $\theta$ cyclic such that $C_{xy}(\theta)$ is cyclic. If $B(\theta) = C_{xy}(\theta)$ we do not need to check anything. Suppose now that $B(\theta) = C_{xy}(D_{xy}(\theta)) \neq C_{xy}(\theta)$. This is only the case if $\theta_1 = D_{xy}(\theta)$ is cyclic and $C_{xy}(\theta_1)$ is acyclic by the definition of $B$. We have shown above that in this case $C_{xy}(D_{xy}(\theta_1))$ contains a cycle, which is $B(\theta)$.
\end{proof}

\subsection{Monotonicity for consolidation}

\begin{lemma} \label{lem:strictcliqueincrease}
The number of cliques is strictly increased by consolidation.
\end{lemma}

\begin{proof}
We simply find an injection and show one new clique. The only cliques affected are the cliques through the edge that is moved. Suppose the vertices of the clique are $(a,b,x_1 , ... , x_l)$ and that the edge $(a,b)$ got moved to $(c,d)$ where $a$ and $c$ are not in the clique respectively. Then we simply inject to the clique $(c,d,x_1 , ... , x_l)$. Now we have injected all cliques, but more is true. There exists a new clique containing the edge $(c,d)$ which has size of the neighbourhood of $c$, which did not exist previously, so in fact we have a strict increase.
\end{proof}

\begin{lemma} \label{lem:numberofaos}
The number of acyclic orientations of $G = H(k,n_0,n_1,...,n_{k-1})$ is

\[ a(G) = k! \cdot \prod _{i = 0} ^{k-1} (i + 1)^{n_i}. \]

\end{lemma}

\begin{proof}
Observe that an acyclic orientation of a graph defines an ordering on the vertices of the graph. For a clique this is a bijection. First we order all the vertices in the k-clique and obtain an acyclic orientations of the clique in $k!$ ways. Then any vertex $x$ not in the clique can be placed in $d(x) + 1$ positions in the ordering amongst its neighbours. Each one of these leads to a unique extension of the acyclic orientation since $x$ and $\mathcal{N}(x)$ form a clique. Each such extension is independent of the others, which gives us:

\begin{equation} \label{eq:1}
 a(G) = k! \cdot \prod _{x \notin K_k} (d(x) + 1) 
\end{equation}
\begin{equation*}
= k! \cdot \prod _{i = 0} ^{k-1} (i + 1)^{n_i}.
\end{equation*}
\end{proof}

\begin{corollary} \label{cor:strictaoincrease}
The number of acyclic orientations is strictly decreased by consolidation.
\end{corollary}

\begin{proof}
Each non-trivial consolidation move replaces two factors in expression (\ref{eq:1}) $s+1$ and $l+1$ by two new factors $s$ and $l+2$. Since $l \geq s$ we have a strict decrease in the number of acyclic orientations following consolidation.
\end{proof}

Note that the number of edges of the line graph of $G$ is not monotone decreasing with respect to consolidation. This is easiest seen by considering a star, which maximises the number of edges in the line graph since any two edges in $G$ share a common vertex. Consolidation takes us away from the star, reducing the number of edges in the line graph.

\subsection{Extremal Graphs}

We would now like to show that the extremal graph is unique in both of these parameters, except for a special case.

\begin{definition}
For any $n$ and $k \le n-2$, the graph $F_{n,k}$ is the union of a complete graph $K_k$, one isolated edge and $n-k-2$ isolated vertices. Observe that $F_{n,k}$ has $n$ vertices and $\binom{k}{2} +1$ edges.
\end{definition}

\begin{theorem}
For graphs with $n$ vertices and $m$ edges, the graph $H_{n,m}$ maximises the number of cliques and minimises the number of acyclic orientations. When $m = \binom{k}{2} +1$ for some integer $k$ the graph $F_{n,k}$ also maximises the number of cliques and minimises the number of acyclic orientations. No other graphs take either of these extremal values.
\end{theorem}

\begin{proof}
Given a graph $G$ we know that the fully compressed $G$ lies in $\mathcal{H}$ by Theorem \ref{thm:compressed}. We know that the final consolidation move takes us to the graph $H_{n,m}$ in Lemma \ref{lem:consolidation_terminates} and that we have a strict increase in the number of cliques and strict decrease in the number of acyclic orientations respectively (Lemma \ref{lem:strictcliqueincrease} and Corollary \ref{cor:strictaoincrease}) for consolidation. Thus we need only show that the compression moves that take us directly to $H_{n,m}$ or $F_{n,m}$ strictly increase or strictly decrease the parameters respectively.

Consider a graph $G$ and a compression such that $C_{xy}(G) = F_{n,k}$. If compression $C_{xy}$ involves the isolated edge in $F_{n,k}$ then this edge is of the form $(v,y)$ and thus $(v,x)$ is an isolated edge in $G$, hence $G \cong F_{n,k}$. Now we can ignore this edge and are left with the graph $H_{n,\binom{k}{2}}$.

Take a graph $G$ and vertices $x$ and $y$ such that $G$ is not in $\mathcal{H}$ but $C_{xy}(G) = H_{n,m}$. The vertices of $H_{n,m}$ consists of the following sets:

\begin{itemize}
  \item $V_0$ - the set of isolated vertices
  \item $w$ - the unique vertex connected to the clique but not in the clique
  \item $V_k$ - the set of vertices in the clique
\end{itemize}

Observe that $y \notin V_0$, as there are no edges in $C_{xy}(G)$ going to the set $V_0$. In addition, $x \notin V_k$, as a compression to another vertex in $V_k$ would have been an isomorphism, and a compression to a vertex in $V_0$ or to vertex $w$ would have moved some other edges too, a contradiction. We are left with the following cases, which we need to check for both graph parameters

\begin{enumerate}

  \item  $x\in V_0$, $y=w$ \label{move_1}
  \item  $x\in  V_0$, $y\in V_k$ \label{move_2}
    \item  $x=w$, $y\in V_k$ with $d(w) \geq 1$ in $C_{wy}(G)$.  \label{move_3}
\end{enumerate}

If no cycles or cliques are destroyed by compression, it is sufficient to show the existence of a new clique for a strict increase in the number of cliques. Furthermore this shows a strict decrease in the number of acyclic orientations, as the new clique adds an extra restriction to the orientations.

Move (\ref{move_1}) does not destroy any cycles or cliques and creates a new clique of size $d(w) +1$ in $C_{xy}(G)$ unless $d(w) = 0$ in $G$ in which case we have an isomorphism. For move (\ref{move_2}) we note that we have a new clique of size $d(w) + 1$, unless $d(w) = 1$ in $C_{xy}(G)$, in which case we could have been compressed from the graph $F_{n,k}$, which has the same number of cliques and acyclic orientations, or we have an isomorphism.

For move (\ref{move_3}) to be a non-isomorphic move there must exist a vertex $z_1 \in V_k$ such that $(z_1, w) \in E(G)$ but $(z_1,y) \notin E(G)$. Since $d(w) \geq 1 $ in $C_{wy}(G)$, there must also exist a vertex $z_2 \in V_k$ such that $(z_2, w), (z_2, y) \in E(G)$. If all vertices in the clique connected to $y$ in $G$ are also connected to $w$, then the compression $C_{wy}(G)$ is an isomorphism which swaps the labels of $w$ and $y$. Hence there exists a vertex $z_3 \in V_k$, $(z_3,y) \in E(G), (z_3,w) \notin E(G). $

Thus for a compression move of type (\ref{move_3}) to be non-isomorphic, the vertices in $V_k$ form a clique of size $k$ in $C_{xy}(G)$ but not in $G$, since $(z_1,y) \notin E(G)$. So by Lemma \ref{lem:compressiondoesnotdecreasecliques} we have a strict increase in the number of cliques.

To show a strict decrease in the number of acyclic orientations we need only find an acyclic orientation in $C_{wy}(G)$ that is mapped to an orientation with a cycle by the map ${B}$, by Corollary \ref{cor:binjsur}. Consider the acyclic orientation $\theta$ defined by the vertex ordering in which the first five vertices are $z_2,y,z_3,z_1,w$ in that order and all other vertices are after this in any order. If the edge $(w,y) \notin G$ then $D_{wy}(\theta)$ is the same as $\theta$ and thus ${B}(\theta) = C_{wy}(\theta)$ which contains the cycle $yz_3z_1$. On the other hand, if $(w,y) \in G$, then both $C_{wy}(\theta)$ and $C_{wy}(D_{wy}(\theta))$ contain the cycle $yz_3z_1$. Thus $B(\theta)$ contains a cycle. 
\end{proof}

\section{Other parameters}

Some other graph parameters, such as the number of connected components and the number
of Euler subgraphs, are also maximised by $H_{n,m}$, but it is not the unique
maximising graph. We give brief arguments for this, independent of the main
thrust of the paper. 

\begin{proposition} \label{prop:connectedcomponents}
Given $n$ and $m$ with $m\le{n\choose2}$, let $k$ be the largest integer such that $m > {k\choose2}$. If $m\neq{k\choose2}+1$, then the number of connected components of a graph with $n$ vertices and $m$ edges is maximised by precisely those graphs with a single component of size $k+1$ and $n-k-1$ isolated
vertices. 

If $m={k\choose2}+1$, then the maximising graph are all of the form of a
complete graph $K_{k}$, one further edge (which may or may not be connected to the
complete graph) isolated vertices (if any). Thus the maximum number of connected components $c$ is $c = n-k$.
\end{proposition}

\begin{proof}
Suppose that $G$ has the maximum number of components. If $G$ has more than two connected components components of size $l,l'\geq 3$ then the edges in connected components of sizes $l$ and $l'$ can be
fitted inside a component of size $l+l'-2$, increasing the number of components
by~$1$, as ${l\choose2}+{l'\choose2}\le{l+l'-2\choose2}$ for $l,l'\geq3$. Thus the maximum is achieved by a graph $G$ with (at most) one component $H$ on more than two vertices and some isolated edges. But two isolated edges can be replaced by two edges joining a single vertex to a component of
size at least $2$, again reducing the number of components; so there is at
most one isolated edge. If $m\neq{{k+1}\choose2}$, then $H$ has a non-edge
which can be replaced by an edge and the isolated edge deleted.
\end{proof}

Note that the graphs in the set $\mathcal{H}$ are of the form described in Proposition \ref{prop:connectedcomponents}, which maximise the number of components. Thus $H_{n,m}$ in particular maximises the number of components of a graph with $n$ vertices and $m$ edges. The same conclusion holds for Euler subgraphs since the number of these
is a monotonic function of the number of connected components, as we now show.

\begin{proposition}
A graph $G$ with $n$ vertices, $m$ edges and $c$ connected components has
$2^{m-n+c}$ Euler subgraphs.
\end{proposition}

\begin{proof}
A subgraph of $G$ can be represented by its characteristic function, a vector
in $\mathbb{Z}_2^m$. The subgraph is Eulerian if and only if the vector is
orthogonal to all vectors of vertex stars in $G$. There are $n$ vertex stars, satisfying
$c$ independent linear relations (the binary sum of the stars of all vertices
in a connected component is zero). So the Euler subgraphs form a subspace of
dimension $m-n+c$.
\end{proof}
\section{Further directions}

There are a number of other graph parameters which are related to those
we have considered here. For example, an orientation of a graph is
\emph{totally cyclic} if it has the property that each edge lies in a
directed cycle. A graph has a totally cyclic orientation if and only
if it is bridgeless. Similarly, there is a lot of interest in the number
of \emph{spanning trees} of a graph; this is non-zero if and only if the
graph is connected.

The number of totally cyclic orientations is not monotonic under compression. 
If we start with a $5$-cycle, compress to a $4$-cycle with a pendant edge,
and compress again to a $4$-cycle with a chord, the number of totally
cyclic orientations goes from $2$ to $0$ and then to $6$.

The number of acyclic orientations, the number of totally cyclic orientations and the number of spanning trees of $G$ are all evaluations of the \emph{Tutte
polynomial} $T_G$, a two-variable polynomial introduced by Tutte in 1947
\cite{Tutte1947}. The Tutte polynomial has a number of specializations of
interest, for example
\begin{itemize}
\item $T(2,0)$ is the number of acyclic orientations \cite{Stanley};
\item $T(1,1)$ is the number of spanning trees;
\item $T(0,2)$ is the number of totally cyclic orientations \cite{Vergnas1980};
\item $T(1,2)$ is the number of spanning subgraphs.
\end{itemize}
See Sokal \cite{Sokal} for a survey.

The number of spanning trees is also one of a class of parameters depending
on the Laplacian eigenvalues of a graph, which are used in statistical design
theory as measures of the optimality of block designs; a design is D-optimal
if it maximizes the number of spanning trees in its concurrence graph. Related
concepts are A-optimality (minimizing the average resistance between vertices,
when the edges of the graph are one-ohm resistors), and E-optimality (which is
connected with isoperimetric number). See \cite{BaileyCameron} for a survey.

There are some problems in applying our methods to these parameters,
apart from the non-monotonicity mentioned above. First,
several of them (such as numbers of totally cyclic orientations and spanning
trees) are zero for disconnected graphs, so the minimization problem is
trivial. For these and other parameters, including the ones we treated in this
paper, we could pose the question:
\begin{question}
Which graphs realize the extremal values of the parameters in question if we
restrict to \emph{connected} graphs?
\end{question}

The big problem is that, for many of these parameters, the extremum in the
other direction is considerably more interesting. Which graphs maximize
the number of acyclic orientations, or of spanning trees, or of totally
cyclic orientations, or minimize the average resistance between pairs of
vertices? These may not be so easy to deal with, but we could ask:
\begin{question}
Is there an ``anti-compression'' move which finds the graphs at the opposite
extreme?
\end{question}

\section{Acknowledgements}
Robert Schumacher was funded by the EPSRC Grant EP/P504872/1.

\bibliographystyle{plain}
\bibliography{Bibtex2}

\end{document}